\documentclass[12pt]{amsart}
\usepackage{geometry}
\usepackage{graphicx}
\usepackage{amssymb}
\usepackage{epstopdf}
\usepackage{amsmath,amscd}
\usepackage{amsthm}
\usepackage{enumerate}
\usepackage{url,verbatim}

\RequirePackage[colorlinks,citecolor=blue,urlcolor=blue]{hyperref}
\usepackage{breakurl}
\theoremstyle{plain}
\makeatletter
\def\l@section{\@tocline{1}{0pt}{1pc}{}{}}
\def\l@subsection{\@tocline{2}{0pt}{1pc}{4.6em}{}}
\def\l@subsubsection{\@tocline{3}{0pt}{1pc}{7.6em}{}}
\renewcommand{\tocsection}[3]{%
  \indentlabel{\@ifnotempty{#2}{\makebox[2.3em][l]{%
    \ignorespaces#1 #2.\hfill}}}#3}
\renewcommand{\tocsubsection}[3]{%
  \indentlabel{\@ifnotempty{#2}{\hspace*{2.3em}\makebox[2.3em][l]{%
    \ignorespaces#1 #2.\hfill}}}#3}
\renewcommand{\tocsubsubsection}[3]{%
  \indentlabel{\@ifnotempty{#2}{\hspace*{4.6em}\makebox[3em][l]{%
    \ignorespaces#1 #2.\hfill}}}#3}
\makeatother

\newtheorem{theorem}{Theorem}
\newtheorem{definition}[theorem]{Definition}
\newtheorem{lemma}[theorem]{Lemma}
\newtheorem{proposition}[theorem]{Proposition}
\newtheorem{corollary}[theorem]{Corollary}

\newtheorem{remark}[theorem]{Remark}

\newcommand\es{\varnothing}

\newcommand\ol{\overline}
\newcommand\Aut{\mathrm{Aut}}

\newcommand\sQ{{\mathcal Q}}
\newcommand\sH{{\mathcal H}}

\newcommand\RR{{\mathbb R}}
\newcommand\ZZ{{\mathbb Z}}

\newcommand\PP{{\mathbb P}}

\renewcommand\a{\alpha}
\newcommand\om{\omega}
\newcommand\g{\gamma}
\newcommand\be{\beta}
\newcommand\si{\sigma}

\newcommand\De{\Delta}
\newcommand\qq{\qquad}
\newcommand\q{\quad}
\newcommand\resp{respectively}

\newcommand\oo{\infty}
\newcommand\sG{{\mathcal G}}

\newcommand\sN{{\mathcal N}}

\newcommand\sM{{\mathcal M}}
\newcommand\Ga{\Gamma}
\newcommand\La{\Lambda}
\newcommand\Si{\Sigma}

\newcommand\la{\lambda}

\newcommand\id{{\bf 1}}

\newcommand\bigmid{\,\big|\,}
\newcommand\pc{p_{\text{\rm c}}}

\renewcommand\ell{l}

\newcommand\pd{\partial}

\newcommand\sm{\,\triangle\,}
\newcommand\ghf{graph height function}
\newcommand\sghf{strong \ghf}
\newcommand\GHF{group height function}
\newcommand\Stab{\mathrm{Stab}}
\newcommand\hdi{$\sH$-difference-invariant}

\renewcommand\o{{\mathrm o}}
\newcommand\normal{\trianglelefteq}

\newcommand\EG{\mathrm{EG}}
\newcommand\EGF{\mathrm{EFG}}  %%{\EG^{\mathrm{F}}}
\newcommand\BS{\mathrm{BS}}
\newcommand\dist{\mathrm{dist}}
\newcommand\Od{\mathrm{ord}}
\newcommand\saw{self-avoiding walk}
\newcommand\pde{\partial_{\text{\rm e}}}

\newcommand\sO{{\mathcal O}}

\newcounter{mycount}
\newenvironment{romlist}{\begin{list}{\rm(\roman{mycount})}%
   {\usecounter{mycount}\labelwidth=1cm\itemsep 0pt}}{\end{list}}

\newenvironment{letlist}{\begin{list}{\rm(\alph{mycount})}%
   {\usecounter{mycount}\labelwidth=1cm\itemsep 0pt}}{\end{list}}

\numberwithin{equation}{section}
\numberwithin{theorem}{section}
\numberwithin{figure}{section}

\title{Self-avoiding walks and amenability}
\author{Geoffrey R.\ Grimmett}
\address{Statistical Laboratory, Centre for
Mathematical Sciences, Cambridge University, Wilberforce Road,
Cambridge CB3 0WB, UK} 
\email{g.r.grimmett@statslab.cam.ac.uk, }
\urladdr{\url{http://www.statslab.cam.ac.uk/~grg/}}

\author{Zhongyang Li}
\address{Department of Mathematics,
University of Connecticut,
Storrs, Connecticut 06269-3009, USA}
\email{zhongyang.li@uconn.edu}
\urladdr{\url{http://www.math.uconn.edu/~zhongyang/}}

\begin{document}
\begin{abstract}
The connective constant $\mu(G)$ of an infinite transitive graph $G$
is the exponential growth rate of the number of self-avoiding walks 
from a given origin.  The relationship between connective constants and amenability
is explored in the current work.

Various properties of connective constants depend on
the existence of so-called \lq\ghf s\rq, namely: (i) whether 
$\mu(G)$ is a local function on certain graphs derived from $G$, 
(ii) the equality of $\mu(G)$
and the asymptotic growth rate of bridges,
and (iii) whether there exists a terminating algorithm for 
approximating $\mu(G)$ to a given degree of accuracy.

In the context  of amenable groups,
it is proved that the Cayley graphs of infinite, finitely generated, elementary amenable
groups support \ghf s, which are in addition harmonic.
In contrast, the Cayley graph of the Grigorchuk group,
which is amenable but not elementary amenable, does not have a \ghf.

In the context of non-amenable, transitive graphs, a lower bound is presented for the connective constant
in terms of the spectral bottom of the graph.  This is a strengthening of
an earlier result of the same authors. 
Secondly, using a percolation inequality of Benjamini, Nachmias, and Peres, it is explained that
the connective constant of a non-amenable, transitive graph 
with large girth is close to that of a regular tree.
Examples are given of non-amenable groups without \ghf s,
of which one is the Higman group. 
\end{abstract}

\date{24 November 2015}   %\today}

\maketitle

\tableofcontents

\part{Background, and summary of results}

\section{Introduction}\label{sec:intro}

\subsection{Background}

A \emph{\saw} on a graph $G=(V,E)$ is a path that visits no vertex more than once.
The study of the number $\si_n$ of \saw s of length $n$ from a given initial vertex was initiated
by Flory \cite{f} in his work on polymerization, and this topic has acquired an iconic status
in the mathematics and physics associated with lattice-graphs. Hammersley and
Morton \cite{hm} proved in 1954 that, if $G$ is vertex-transitive, there exists a constant $\mu=\mu(G)$,
called the \emph{connective constant} of $G$, 
such that $\si_n=\mu^{n(1+\o(1))}$ as $n \to\oo$.
This result is important not only for its intrinsic value, but also because its proof 
contained the introduction of
subadditivity to the theory of interacting systems. Subsequent work has concentrated on
understanding polynomial corrections in the above asymptotic for $\si_n$
(see, for example, \cite{bdgs,ms}), and on finding exact values and inequalities for connective constants
(for example, \cite{ds,GrLrev}).

There are several natural questions about connective constants whose answers depend on 
whether or not the underlying graph admits a so-called \ghf.
The first of these is whether $\mu(G)$ is a continuous
function of the graph $G$ (see \cite{Ben13,GL-loc}). This so-called \emph{locality} question
has received serious attention also in the context of percolation and other disordered systems
(see \cite{bnp,mt,psn}),
and has been studied in recent work of the current authors on general transitive
graphs, \cite{GL-loc},  and also on Cayley graphs of finitely generated groups, \cite{GL-Cayley}.
Secondly, when $G$ has a \ghf, one may define 
\emph{bridge} \saw s on $G$, and show that their numbers grow asymptotically
in the same manner as $\si_n$ (see \cite{GL-loc}). 
The third such question is whether there exists a terminating
algorithm to approximate $\mu(G)$ within any given (non-zero) margin of accuracy
(see \cite{GL-loc, GrL3}).

Roughly speaking, a \ghf\ on $G=(V,E)$ is a non-constant 
function $h:V\to\ZZ$ whose increments are invariant
under the action of a finite-index subgroup of automorphisms (a formal definition may be found  
at Definition \ref{def:height}).  It is, therefore, useful to know which transitive graphs support \ghf s. 

A method for constructing \ghf s on a certain class of
transitive graphs is described in \cite{GL-loc},
and the question is posed there of deciding whether all transitive graphs support \ghf s.
A rich source of interesting examples of transitive graphs is provided by Cayley graphs of finitely
generated groups, as studied in \cite{GL-Cayley}. It is proved there that the Cayley graphs
of finitely generated, virtually solvable groups support \ghf s, which are in addition harmonic.
The question is posed of determining whether or not the Cayley graph of the Grigorchuk group
possesses a  \ghf.

We are concerned here with the relationship between connective constants
and amenability, and we present results for both amenable and for non-amenable graphs.
Since these results are fairly distinct, we summarize them here under the
two headings of amenable groups and non-amenable graphs.

\subsection{Amenable groups}

This part of the current work has two principal results, one positive and the other negative.
\begin{letlist}
\item
(Theorem \ref{eag}) It is proved that every Cayley graph of an infinite, finitely generated, 
elementary amenable group supports a \ghf. This extends \cite[Thm 5.1]{GL-Cayley}
beyond the class of virtually solvable groups. 

\item
(Theorem \ref{grig}) It is proved that
the Cayley graph of the Grigorchuk group does not support a \ghf. This answers 
in the negative the above question of \cite{GL-loc} (see also \cite[Sect.\ 5]{GL-Cayley}).
Since the Grigorchuk group is amenable (but not elementary amenable), 
possession of a  \ghf\ is not a characteristic of amenable groups.
This is in contrast with work of Lee and Peres, \cite{LeeP}, who have studied the existence of
non-constant, Hilbert space valued, equivariant harmonic
maps on amenable graphs.
\end{letlist}

\subsection{Non-amenable graphs}

In earlier work \cite{GL-Comb}, it was shown that the connective constant $\mu$ of
 a transitive, simple graph with degree $\De$ satisfies
 $$
 \sqrt{\De-1}\le \mu \le \De -1,
 $$
 and it was asked whether or not the lower bound is sharp. 
 In the first of the following three results, this
 is answered in the negative for non-amenable graphs.
 
 \begin{letlist}
 \item
 (Theorem \ref{thm:newin}) It is proved that
 $$
 (\De-1)^{\frac12(1+c\lambda)} \le \mu,
 $$
 where $c=c(\De)$ is a known constant, and $\lambda$ is the spectral bottom of the simple
 random walk on the graph. Kesten \cite{K59b,K59a} and Dodziuk \cite{Dod}
 have shown that $\lambda>0$ if and only if the graph is  non-amenable.
 
 \item
 (Theorem \ref{thm:pc}) Using a percolation result of Benjamini, Nachmias, and Peres \cite{bnp}, it is explained
 that the connective constant of a non-amenable, $\De$-regular graph with large girth is
 close to that of the $\De$-regular tree.
 
 \item 
 (Theorem \ref{thm:hg1})
 It is shown that the Cayley graph of the Higman group of \cite{Hg} (which is non-amenable) does
 not support a \ghf. This further example of a transitive graph without a \ghf\ complements the corresponding
 statement above for the (amenable) Grigorchuk group.
 
 \end{letlist}
 
Relevant notation for graphs, groups, and self-avoiding walks is summarized
in Section \ref{sec:Cay}, and three different
types of height functions are explained in Section \ref{sec:ht}.
The class $\EG$ of elementary amenable groups
is described in Sections \ref{sec:eag} and \ref{sec:proof1}. 
The Grigorchuk group is defined
in Section \ref{sec:grig} and the non-existence of \ghf s thereon is given in Theorem \ref{grig}.
The improved lower bound for $\mu(G)$ for non-amenable $G$ is presented at Theorem \ref{thm:newin},
and the remark about non-amenable graphs with large girth at Theorem \ref{thm:pc}.
The Higman group is discussed in Section \ref{sec:hig}. Proofs of theorems appear either
immediately after their statements, or are deferred to self-contained sections at the end of the article.

\section{Graphs, groups, and self-avoiding walks}\label{sec:Cay}

\subsection{Graphs}

The graphs $G=(V,E)$ in this paper are \emph{simple}, 
in that they have neither loops nor multiple edges. 
The \emph{degree} $\deg(v)$ of vertex $v\in V$ is the number of edges incident to $v$.
We write $u \sim v$ for neighbours $u$ and $v$,  
$\pd v$ for the neighbour set of $v$, and $\pde v$ (\resp, $\pde W$) for the set of edges incident to $v$
(\resp, between $W$ and $V \setminus W$).
The graph is \emph{locally finite} if $|\pd v|<\oo$ for $v \in V$.  
An edge from $u$ to $v$ is denoted $\langle u,v\rangle$
when undirected, and $[u,v\rangle$ when directed from $u$ to $v$.
The \emph{girth} of $G$ is the infimum of the lengths of its circuits.
The infinite $\De$-regular tree $T_\De$ crops up periodically in this paper.

The automorphism group of  $G$ is
denoted $\Aut(G)$. The subgroup $\Ga\le\Aut(G)$ is said to act \emph{transitively} on $G$
if, for $u,v\in V$, there exists
$\a\in\Aut(G)$ with $\a(u)=v$. 
It acts \emph{quasi-transitively} if there exists a finite
subset $W \subseteq V$ such that, for $v\in V$, there exists $\a\in\Ga$ and $w\in W$
such that $\a(v)=w$.
The graph $G$ is said to be \emph{(vertex-)transitive} if $\Aut(G)$ acts transitively
on $V$.

Let $\sG$ be the set of infinite, locally finite, connected, transitive, simple graphs, and let $G \in \sG$.
The \emph{edge-isoperimetric constant} $\phi=\phi(G)$ is defined here as
\begin{equation}\label{eq:iso}
\phi :=\inf\left\{\frac{|\pde W|}{\De|W|}: W \subset V,\ 0<|W|<\oo\right\}.
\end{equation}
We call $G$ \emph{amenable} if $\phi=0$ and \emph{non-amenable} otherwise.
See \cite[Sect.\ 6]{LyP} for an account of graph amenability.

\subsection{Self-avoiding walks}

Let $G \in \sG$. We choose a vertex of $G$ and call it the \emph{origin}, denoted $\id$.
An \emph{$n$-step self-avoiding walk} (SAW) 
on $G$ is  a walk containing $n$ edges
no vertex of which appears more than once.
Let $\Si_n$ be the set of $n$-step SAWs starting at $\id$, with
cardinality $\si_n:=|\Si_n|$. 
We have in the usual way (see \cite{hm,ms}) that
\begin{equation}\label{eq:sisub}
\si_{m+n} \le \si_m\si_n,
\end{equation}
whence the \emph{connective constant}
\begin{equation}\label{eq:sisub2}
\mu=\mu(G) :=\lim_{n\to\oo} \si_n^{1/n}
\end{equation}
exists. 

A SAW is called \emph{extendable} if it is the initial portion of an infinite SAW on $G$.
(An extendable SAW is called `forward extendable' in \cite{GHP}.)

\subsection{Groups}

Let $\Ga$ be a group with
generator set $S$ satisfying $|S|<\oo$ 
and $\id\notin S$, where
$\id=\id_\Ga$ is the identity element.  We shall assume that $S^{-1}=S$,
while noting that this  was not assumed in \cite{GL-Cayley}.
We write $\Ga =\langle S \mid R\rangle$ with $R$ a set of relators
(or relations, when convenient).
Such a group is called \emph{finitely generated}, and is called \emph{finitely presented}
if, in addition,  $|R|<\oo$.

The \emph{Cayley graph} of the presentation 
$\Ga=\langle S \mid R\rangle$ is the simple graph $G=G(\Ga,S)$
with vertex-set $\Ga$, and an (undirected) edge $\langle \g_1,\g_2\rangle$ if and only
if $\g_2=\g_1s$ for some $s\in S$. Thus, our Cayley graphs are simple graphs.
See \cite{bab95,LyP} for accounts of Cayley graphs, and \cite{dlH} of geometric group theory.

The amenability of groups was introduced by von Neumann \cite{vN}.
It is standard that a finitely generated group is amenable if and only if
some (and hence every) Cayley graph is amenable
(see, for example, \cite[Chap.\ 12A]{Wo}).

\section{Height functions}\label{sec:ht}

It was shown in \cite{GL-loc} that graphs $G \in\sG$ supporting 
so-called `\ghf s' have (at least) three properties:
\begin{romlist}
\item one may define the concept of a `bridge' SAW on $G$, as in \cite{HW62},
\item the exponential growth rate for counts of bridges equals the connective constant $\mu(G)$,
\item there exists a terminating algorithm for determining $\mu(G)$ to within
any prescribed (strictly positive) degree of accuracy.
\end{romlist} 

Several natural sub-classes of $\sG$ contain only graphs that support \ghf s, and
it was asked in \cite{GL-Comb} whether or not \emph{every} $G \in \sG$ supports
a \ghf. This question will be answered in the negative at Theorems \ref{grig} and \ref{thm:hg1},
where it is proved that neither the Grigorchuk nor Higman graphs possess a \ghf.
Arguments for proving the non-existence of \ghf s may be found in Section \ref{sec:noht}.

We review the definitions of the two types of height functions, and introduce a third type.
Let $G=(V,E)\in\sG$, and let $\sH \le\Aut(G)$.
A function $F:V \to \RR$ is said to be \emph{\hdi} if 
\begin{equation}\label{eq:hdi2}
F(v)-F(w)=F(\g v)-F(\g w), \qq v,w\in V,\ \g\in \sH.
\end{equation}

\begin{definition}[\cite{GL-loc}] \label{def:height}

A \emph{graph height function} on $G$ is a pair $(h,\sH)$, where 
$\sH\le \Aut(G)$ acts quasi-transitively on $G$ and  $h:V \to \ZZ$, such that:
\begin{letlist}
\item $h(\id)=0$,
\item $h$ is \hdi,
\item for  $v\in V$,
there exist $u,w \in \pd v$ such that
$h(u) < h(v) < h(w)$.
\end{letlist}
\end{definition}

\begin{remark}\label{rem:Poin}
By Poincar\'e's Theorem for subgroups (see \cite[p.\ 48, Exercise 20]{Her}), it is immaterial
whether or not we require the subgroup $\sH$ to be \emph{normal} in Definition \ref{def:height}.
\end{remark}

We turn to Cayley graphs of finitely generated groups.
Let $\Ga$ be a finitely generated group with presentation $\langle S\mid R\rangle$.
As in Section \ref{sec:Cay}, we assume $S^{-1}=S$ and $\id \notin S$.

\begin{definition}\label{def:GHF}
A \emph{\GHF}\ on $\Ga$ (or on a Cayley graph of $\Ga$) is a function $h:\Ga\to\ZZ$ such that:
\begin{letlist}
\item $h(\id)=0$, and $h$ is not identically zero,
\item if $\g=s_1s_2\cdots s_m$ with $s_i \in S$, then $h(\g)=\sum_{i=1}^m h(s_i)$,
\item the values $(h(s): s\in S)$ are such that, if $s_1s_2\cdots s_n=\id$
is a representation of the identity with $s_i\in S$, then $\sum_{i=1}^n h(s_i)=0$.  
\end{letlist}
\end{definition}

A necessary and sufficient condition for the existence of a \GHF\ is given in \cite[Thm 4.1]{GL-Cayley}.
In the language of group theory, this condition amounts to requiring that the first Betti number is 
strictly positive.
It was recalled in \cite[Remark 4.2]{GL-Cayley} that
(when the non-zero $h(s)$, $s\in S$, are coprime) a \GHF\ is simply a surjective homomorphism
from  $\Ga$ to $\ZZ$.

We introduce a third type of height function, which may be 
viewed as an intermediary between a  \ghf\ and \GHF.

\begin{definition}\label{def:sghf}
For a Cayley graph $G$ of
a finitely generated group $\Ga$, we say that the pair $(h,\sH)$ is
a \emph{\sghf} of the pair $(\Ga,G)$ if
\begin{letlist}
\item $\sH\normal \Ga$ acts on $\Ga$ by left multiplication, and $[\Ga:\sH]<\oo$, and
\item $(h,\sH)$ is a \ghf.
\end{letlist}
\end{definition}

It is evident that a \GHF\ $h$ (of $\Ga$)  is a \sghf\ of the form $(h,\Ga)$, and a \sghf\ is a \ghf.
The assumption in (a) above of the normality of $\sH$ is benign, as in Remark \ref{rem:Poin}.  

We recall the definition  of a harmonic function. A function $h:V \to\RR$ is called \emph{harmonic}
on the graph $G=(V,E)$ if
$$
h(v) = \frac1{\deg(v)}\sum_{u \sim v} h(u), \qq v \in V.
$$
It is an exercise to show that any \GHF\ is harmonic.

\part{Results for amenable groups}

\section{Elementary amenable groups}\label{sec:eag}

The class $\EG$ of elementary amenable groups was introduced by Day in 1957, \cite{Day57}, as the 
smallest class of groups that contains the set $\EG_0$ of all finite and abelian groups, 
and is closed under the operations 
of taking subgroups, and of forming quotients, extensions, and directed unions.
Day noted that every group in $\EG$ is amenable (see also von Neumann \cite{vN}). 
An important example of an amenable but not elementary amenable group was described by 
Grigorchuk in 1984, \cite{Grig84}. 
Grigorchuk's group is important in the study of height functions, and
we return to this in Section \ref{sec:grig}.

Let $\EGF$ be the set of
infinite, finitely generated members of $\EG$. 

\begin{theorem}\label{eag}
Let $\Ga\in\EGF$. Any locally finite Cayley graph $G$ of $\Ga$ admits a harmonic, strong \ghf.
\end{theorem}

We prove a slightly stronger version of this at Theorem \ref{st},
using transfinite induction.
The class $\EGF$ includes all virtually solvable groups, and thus Theorem \ref{eag}
extends \cite[Thm 5.1]{GL-Cayley}.
Since any finitely generated group with polynomial growth is virtually nilpotent, \cite{MG81},
and hence lies in $\EGF$, its locally finite Cayley graphs admit harmonic \ghf s.

\section{The Grigorchuk group}\label{sec:grig}

The (first) Grigorchuk group is an infinite, finitely generated, amenable group that is not elementary amenable.
We show in Theorem \ref{grig} that there exists a 
locally finite Cayley graph of the Grigorchuk group with \emph{no} \ghf.
This answers in the negative Question 3.3 of \cite{GL-loc} (see also \cite[Sect.\ 3]{GL-Cayley}).
 
Here is the definition of the group in question (see \cite{Grig80,Grig84,RG05}).
Let $T$ be the rooted binary tree with root vertex $\es$. 
The vertex-set of $T$ can be identified with the set of finite strings $u$ having entries $0$, $1$, 
where the empty string corresponds to the root $\es$.
Let $T_u$ be the subtree of all vertices with root labelled $u$. 

Let $\Aut(T)$ be the automorphism group of $T$, and
let $a\in\Aut(T)$ be the automorphism
that, for each string $u$, interchanges the two vertices $0u$ and $1u$. 

Any $\g\in\Aut(G)$ may be applied
in a natural way to either subtree $T_i$, $i=0,1$. 
Given two elements 
$\g_0,\g_1\in\Aut(T)$, we define $\g=(\g_0,\g_1)$ to be the automorphism
on $T$ obtained by applying $\g_0$ to $T_0$ and $\g_1$ to $T_1$.
Define automorphisms $b$, $c$, $d$ of $T$ recursively as follows:
\begin{equation}\label{eq:grigrels}
b=(a,c),\quad c=(a,d),\quad d=(e,b),
\end{equation}
where $e$ is the identity automorphism. 
The Grigorchuk group is defined as the subgroup of  $\Aut(T)$ generated by the 
set $\{a,b,c\}$.

\begin{theorem}\label{grig}
The Cayley graph $G=(V,E)$ of the Grigorchuk group with generator set $\{a,b,c\}$ satisfies:
\begin{letlist}
\item $G$  admits no \ghf,
\item for $\sH\normal\Aut(G)$ with finite index, 
any \hdi\  function on $V$ is constant on each orbit of $\sH$.
\end{letlist}
\end{theorem}

The proof of Theorem \ref{grig} is given in Section \ref{sec:pfgrig}.
In the preceding Section \ref{sec:noht}, two approaches are developed for
showing the absence of a \ghf\ within particular classes of Cayley graph. 
In the case of the Grigorchuk group, two reasons combine to forbid \ghf s,
namely, its Cayley group has no automorphisms beyond the action of the group itself, and 
the group is a torsion group in that every element  has finite order.

Since the Grigorchuk group is amenable,  Theorems \ref{eag} and \ref{grig} yield that:
within the class of infinite, finitely generated groups, every elementary amenable group has a \ghf, but there
exists an amenable group without a \ghf.  The Grigorchuk group is finitely \emph{generated} but not finitely 
\emph{presented}, \cite[Thm 6.2]{Grig84}.

We ask if there exists an infinite, finitely \emph{presented}, amenable group 
with a Cayley graph having no \ghf.
A natural candidate might be the group $\Ga=\langle S\mid R\rangle$ of \cite[Thm 1]{Grig98},
with
\begin{gather*}
S=\{a,c,d,t\},\\
R=\bigl\{a^2=c^2=d^2=(ad)^4=(adacac)^4=\id,
t^{-1}at=aca, t^{-1}ct=dc, t^{-1}dt=c\bigr\}.
\end{gather*}
This finitely presented, amenable HNN-extension of the Grigorchuk group 
is not elementary amenable.
However, since it contains  the free group generated by the stable letter $t$, it possesses a
\GHF. More precisely, the function
$$
h(\id)=0,\q h(t)=1,\q h(t^{-1})=-1,\q \q h(s)=0 \text{ for } s\in S,\ s\ne t^{\pm 1},
$$
defines  a \GHF.

\part{Results for non-amenable graphs}

\section{Connective constants of non-amenable graphs}\label{sec:nonamen}

Let $G\in\sG$ have degree $\De$.
It was proved in \cite[Thm 4.1]{GL-Comb} that
$$
\sqrt{\De-1} \le \mu(G) \le \De -1.
$$
The upper bound is achieved by the $\De$-regular tree $T_\De$. It is unknown
if the lower bound is sharp for \emph{simple} graphs. This lower bound may however be improved
for non-amenable graphs, as follows.

Let $P$ be the transition matrix of simple random walk (SRW) on $G=(V,E)$, and let $I$ be the 
identity matrix. The \emph{spectral bottom} of $I-P$ is defined to be the largest  
$\lambda$ with the property that, for all $f\in\ell^2(V)$, 
\begin{equation}\label{eq:spbo}
\langle f, (I-P)f \rangle\geq \lambda\langle f,f\rangle.
\end{equation}
It may be seen that $\lambda(G)=1-\rho(G)$ where $\rho(G)$ is
the spectral radius of $P$ (see \cite[Sect.\ 6]{LyP}, and \cite{Wo} for
an account of the spectral radius).

\begin{remark}\label{rem0}
It is known that $G$ is a non-amenable if and only if $\rho(G)<1$,
which is equivalent to $\lambda(G)>0$. This was proved
by Kesten \cite{K59b, K59a} for Cayley graphs of finitely-presented groups, and 
extended to general transitive graphs by  Dodziuk \cite{Dod}
(see also the references in \cite[Sect.\ 6.10]{LyP}).
\end{remark}

\begin{theorem}\label{thm:newin}
Let $G\in \sG$ have degree $\De \ge 3$. 
Let $P$ be the transition matrix of SRW on $G$, and $\lambda$ 
the spectral bottom of $P$. The connective constant $\mu(G)$ satisfies 
\begin{equation}\label{eq:newin}
\mu(G)\geq (\De-1)^{\frac12(1+c\lambda)},
\end{equation}
where $c=\De(\De-1)/(\De-2)^2$.
\end{theorem}

The improvement in the lower bound for $\mu(G)$ is strict if and only if $\lambda>0$, 
which is to say that $G$ is non-amenable. It is standard (see \cite[Thm 6.7]{LyP}) that
\begin{equation}\label{eq:lambdaineq}
\tfrac 12\phi^2 \le 1-\sqrt{1-\phi^2} \le \lambda \le \phi,
\end{equation}
where $\phi=\phi(G)$ is the edge-isoperimetric constant of \eqref{eq:iso}.
By \cite[Thm 3]{AGV},
\begin{equation}\label{eq:univ}
\lambda(G) \le \lambda(T_\De)-\frac{\De-2}{\De(\De-1)^{g+2}},
\end{equation}
where $g$ is the girth of $G$, $T_\De$ is the $\De$-regular tree, and  
\begin{equation}\label{eq:tree}
\lambda(T_\De)=1-\frac{2\sqrt{\De-1}}\De.
\end{equation}

\begin{remark}\label{rem1}
The spectral bottom (and therefore the spectral radius, also) is
not a continuous function of $G$ in the usual graph metric
(see \cite[Sect.\ 5]{GL-loc}). This
follows from \cite[Thm 2.4]{NS}, where it is proved that, for all pairs $(k,l)$ with $k \ge 2$ and $l \ge 3$,
there exists a group with polynomial growth whose Cayley
graph $G_{k,l}$ is $2k$-regular with girth exceeding $l$. Since $G_{k,l}$
is amenable, we have $\lambda(G_{k,l})=0$, whereas $\lambda(T_{2k})$ is given by \eqref{eq:tree}.
\end{remark}

\begin{proof}[Proof of Theorem \ref{thm:newin}]

This is achieved by a refinement of the argument used to prove \cite[Thm 4.1]{GL-Comb},
and we shall make use of the notation introduced in that proof.

Let $v_0=\id$, and let $\pi=(v_0,v_1,\dots,v_{2n})$ be an extendable $2n$-step SAW of $G$.
For convenience, we augment $\pi$ with a mid-edge incident to $v_0$ and not lying on
the edge $\langle v_0,v_1\rangle$. 
Let $E_{\pi}$ be the set of oriented edges $[v,w\rangle$ such that: (i) $v\in \pi$,  $v \ne v_{2n}$,
and (ii) the (non-oriented) edge $\langle v,w\rangle$ does not lie in $\pi$. 
Note that 
\begin{equation}\label{l0}
|E_{\pi}|=2n(\De-2).
\end{equation}

Each (oriented) edge in $E_{\pi}$ is coloured either red or blue according to the following rule. 
For $v\in \pi$, let $\pi_v$ be the sub-path of $\pi$ joining $v_0$ to $v$. 
The edge $[v,w\rangle\in E_{\pi}$ is coloured red if $\pi_v\cup[v, w\rangle$ 
is not an extendable SAW, and is coloured blue otherwise. By \eqref{l0}, 
the number $B_\pi$ (\resp, $R_\pi$) of blue edges
(\resp, red edges) satisfies
\begin{equation}\label{l-1}
B_\pi+R_\pi = 2n(\De-2).
\end{equation}
We shall make use of the following lemma.

\begin{lemma}\label{lem1}
The number $B_\pi$ satisfies
\begin{equation}\label{l4}
B_\pi \ge \frac {n(1+c\lambda)}{\De-2}  -\frac {\De-1} 2,
\end{equation}
where $c=\De(\De-1)/(\De-2)^2$.
\end{lemma}

We now argue as in \cite[Lemma 5.1]{GL-Comb} to deduce from Lemma \ref{lem1} 
that the number of extendable $2n$-step
SAWs from $v_0$ is at least $C(\De-1)^{n(1+c\lambda)}$ where $C=C(\De)$.  Inequality \eqref{eq:newin} follows as required.
\end{proof}

\begin{proof}[Proof of Lemma \ref{lem1}]
An edge $[v,w\rangle \in E_\pi$ is said to be \emph{finite} if $w$ lies in a finite
component of $G \setminus \pi$, and \emph{infinite} otherwise.
If $[v,w\rangle\in E_{\pi}$ is red, then $w$ is necessarily finite. 
Blue edges, on the other hand, may be either finite or infinite.

It was explained in the proof of \cite[Thm 4.1]{GL-Comb} that there exists an injection 
$f$ from the set of red edges to
the set of blue edges with the property that, if $e=[v,w\rangle$ is red, and 
$f(e)=[v',w'\rangle$, then $w$ and $w'$ lie in the same component of
$G\setminus \pi_v$. Since $e$ is finite, so is $f(e)$. It follows that
\begin{equation}\label{eq:inf}
B_\pi \ge R_\pi+B^\oo_\pi,
\end{equation}
where $R_\pi$ is the number of red edges, and $B^\oo_\pi$ is the number of infinite blue edges. 

Let $X=(X_m: m=0,1,2,\dots)$ be a SRW on $G$, and let $\PP_v$ denote 
the law of $X$ started at $v\in V$.
For $[ v,w \rangle\in E_{\pi}$, let 
\begin{equation*}
\beta_{[v,w\rangle}=\PP_v\bigl(X_1=w,\text{ and } \forall m>0, X_m\notin \pi\bigr).
\end{equation*}
By \cite[Lemma 2.1]{bnp} with $A=\pi$,
\begin{equation}\label{l1}
\lambda \le \frac{1}{2n+1}\left(\sum_{[v,w\rangle\in E_{\pi}}\beta_{[v,w\rangle} + 
\sum_w \beta_{[v_{2n},w\rangle}\right).
\end{equation}
If $[v,w\rangle\in E_\pi$ is finite, then $\beta_{[v,w\rangle}=0$. By \eqref{l1},
\begin{equation}\label{l20}
\lambda \le \left(\frac{ B_\pi^\oo+\De-1}{2n}\right)
\PP_v\bigl(X_1=w,\text{ and } \forall m>0, X_m\ne v\bigr).
\end{equation}
The last probability depends on the graph $G$, 
and it is a maximum when $G$ is the $\De$-regular tree $T_\De$ (since $T_\De$ is
the universal cover of $G$).
Therefore, it is no greater than $1/\De$ multiplied by the probability that a random walk on $\ZZ$,
which moves rightwards with probability $p=(\De-1)/\De$ and leftwards
with probability $q=1-p$, never visits $0$ having started at $1$.
By, for example, \cite[Example 12.59]{GW},
$$
\PP_v\bigl(X_1=w,\text{ and } \forall m>0, X_m\ne v\bigr) \le \frac1\De \left(1-\frac qp\right)
=\frac{\De-2}{\De(\De-1)}.
$$
By \eqref{eq:inf} and \eqref{l20},
\begin{equation}\label{l21}
B_\pi \ge  R_\pi + 2n\lambda \frac{\De(\De-1)}{\De-2}-\De+1,
\end{equation}
and \eqref{l4} follows by \eqref{l-1}.
\end{proof}

\section{Graphs with large girth}\label{sec:girth}

Benjamini, Nachmias, and Peres showed in \cite[Thm 1.1]{bnp} that the critical probability $\pc(G)$
of bond percolation on a $\De$-regular, non-amenable graph $G$ with large girth is close to that of the critical
probability of the $\De$-regular tree $T_\De$. Their main result implies the following. 

\begin{theorem}\label{thm:pc}
Let $G\in\sG$ be non-amenable with degree $\De\ge 3$ and girth $g\le\oo$. 
There exists an absolute positive constant $C$ such that
\begin{equation}\label{eq:ineq2}
\left[\frac1{\De-1}+C\frac{\log(1+\lambda^{-2})}{g\De}\right]^{-1}
\le \mu(G) \le \De-1,
\end{equation}
where $\lambda=\lambda(G)$ is the spectral bottom of SRW on $G$, as in \eqref{eq:spbo}.
Equality holds in the upper bound of \eqref{eq:ineq2} if and only if $G$ is a tree, that is, $g=\oo$.
\end{theorem}

\begin{proof}
The upper bound of \eqref{eq:ineq2} is from \cite[Thm 4.2]{GL-Comb}. The lower
bound is an immediate consequence of \cite[Thm 1.1]{bnp} and the fact that
$\mu(G) \ge 1/\pc(G)$  (see, for example, \cite[Thm 7]{BH} and \cite[eqn (1.13)]{G99}, which hold
for general quasi-transitive graphs). 
\end{proof}

We recall from Remark \ref{rem0} that $\lambda>0$ if and only if $G$ is non-amenable.
Theorem \ref{thm:pc} does not, of itself, imply that $\mu(\cdot)$ is
continuous at $T_\De$, since $\lambda(\cdot)$ is not continuous at $T_\De$ (in the case
when $\De$ is even, see Remark \ref{rem1}).
For continuity at $T_\De$, it would suffice that $\lambda(\cdot)$ is bounded
away from $0$ on a neighbourhood of $T_\De$. By \eqref{eq:lambdaineq},
this is valid within any class of graphs whose edge-isoperimetric constants \eqref{eq:iso}
are bounded uniformly from $0$. See also \cite[Thm 5.1]{GL-loc}.

\section{The Higman group}\label{sec:hig}

The \emph{Higman group}  $\Ga$  of \cite{Hg} is the infinite, finitely presented 
group with presentation $\Ga=\langle S \mid R\rangle$ where
\begin{equation}\label{eq:higpres}
\begin{aligned}
S&=\{ a,b,c,d,a^{-1},b^{-1},c^{-1},d^{-1}\},\\ 
R&= \{a^{-1}ba=b^{2},b^{-1}cb=c^{2},c^{-1}dc=d^{2}, d^{-1}ad= a^{2}\}.
\end{aligned}
\end{equation}
This group is interesting since it has no proper normal subgroup with finite index, and the
quotient of $\Ga$ by its maximal proper normal 
subgroup is an infinite, finitely generated, simple group. 
By \cite[Thm 4.1(b)]{GL-Cayley}, $\Ga$ has no \GHF.
The above two reasons conspire to forbid \ghf s.

\begin{theorem}\label{thm:hg1}
The Cayley graph $G=(V,E)$ of the Higman group $\Ga=\langle S\mid R\rangle$ has no \ghf.
\end{theorem}

A further group of Higman type is given as follows. 
Let $S$ be as above, and let $\Ga'=\langle S\mid R' \rangle$ be the finitely presented group with
\begin{equation*}
R'=\{a^{-1}ba=b^{2},b^{-2}cb^2=c^{2},c^{-3}dc^3=d^{2}, d^{-4}ad^4= a^{2}\}.
\end{equation*}
Note that $\Ga'$ is infinite and non-amenable, since the subgroup generated by 
the set $\{a,c,a^{-1},c^{-1}\}$  is a free group (as in the corresponding
step for the Higman group at \cite[pp.\ 62--63]{Hg}).

\begin{theorem}\label{thm:hg2}
The Cayley graph $G=(V,E)$ of the above group $\Ga'=\langle S\mid R'\rangle$ has no \ghf.
\end{theorem}

The proofs of the above theorems are given in Sections \ref{sec:pfhg} and \ref{sec:pfhig}, \resp.

\part{Remaining proofs}

\section{Proof of Theorem \ref{eag}}\label{sec:proof1}

We shall prove the following stronger form of Theorem \ref{eag}. 

\begin{theorem}\label{st}
Let $\Ga\in\EGF$. 
There exists a normal subgroup $\sH\normal\Ga$
with $1<[\Ga:\sH]<\oo$ such that any locally finite Cayley graph $G$ of 
$\Ga$ possesses a harmonic, \sghf\ of the form $(h,\sH)$.
\end{theorem}
 
Whereas every member of $\EGF$ has a proper, normal subgroup
with finite index, it is proved in \cite{JM} that there exist amenable \emph{simple}
groups.

We review next the structure of $\EG$. Let $\EG_0$ be the class of all groups
that are either finite or abelian (or both), and let $\sO$ be the class of all ordinals. 
Let $\alpha\in\sO$, $\a\ne 0$, and assume we have defined $\EG_{\beta}$ for each $\be\in\sO$,
$\beta<\alpha$. Each $\a\in\sO$ is either a limit ordinal or a successor ordinal.
If $\alpha$ is a limit ordinal, we set 
\begin{equation}\label{eq:limitord}
\EG_{\alpha}=\bigcup_{\be<\a} \EG_{\beta}.
\end{equation}
If $\alpha$ is a successor ordinal, let $\EG_{\alpha}$ be the class of groups 
which can be obtained from members of $\EG_{\alpha-1}$ by no more
than  one operation of extension or directed union.

\begin{theorem}[\cite{CH80}]\label{lem:cc}
We have that $\EG=\bigcup_{\a\in\sO} \EG_{\alpha}$.
\end{theorem}

\begin{proof}[Proof of Theorem \ref{st}]

Let $\EGF_\a=\EGF \cap \EG_\a$.  
For $\a\in\sO$, let H$_\a$ be the following statement:
\begin{itemize}
\item [$\text{H}_\a$\ :] for $\be\in\sO$, $\beta\le\alpha$, and $\Ga\in \EGF_{\beta}$, there exists
$\sH\normal\Ga$ such that every locally finite Cayley graph of $\Ga$ admits
a harmonic, \sghf\ of the form $(h,\sH)$.
\end{itemize}

Now, $\EGF_0$ is the set of infinite, finitely generated, abelian groups. 
By \cite[Prop.\ 4.3, Thm 5.2(b)]{GL-Cayley}, any locally finite Cayley graph of $\Ga$ has a 
\GHF, and hence a harmonic, \sghf\ of the form $(h,\Ga)$. Therefore, H$_0$ holds, and we turn to the 
induction step.

Let $\a\in\sO$, $\a\ne 0$, and assume H$_\be$ holds for all $\be< \a$. 
Let $\Ga\in \EGF_{\alpha}$ with $\a$ the smallest such ordinal.
There are two cases to consider, depending on whether or not $\a$ is a limit ordinal.
\emph{If $\alpha$ is a limit ordinal}, by \eqref{eq:limitord},
there exists $\be\in\sO$, $\be < \a$, such that $\Ga\in\EGF_\be$. 
The claim now follows by H$_\be$. 

\emph{We assume for the remainder of this proof that $\a$ is a successor ordinal.}
By Theorem \ref{lem:cc}, the group $\Ga\in\EGF_\a$ is obtained from groups in $\EGF_{\alpha-1}$ by 
exactly one operation of either extension or directed union. 
That is, there are two sub-cases to consider.
\begin{letlist}
\item There exist $\sN',\sQ'\in \EGF_{\alpha-1}$ such that $\sN'$ is isomorphic to a 
normal subgroup $\sN$ of $\Ga$, and $\sQ'\simeq \sQ:=\Ga/\sN$.
\item There exist a directed set $\La$ and a family $(S_{\lambda}:\lambda\in \Lambda)$ 
satisfying 
\begin{romlist}
\item $S_\la\in\EGF_{\a-1}$, 
\item $S_{\lambda_1}\subseteq S_{\lambda_2}$ whenever $\la_1\le\la_2$, 
\item $\Ga=\bigcup_{\lambda\in \Lambda}S_{\lambda}$.
\end{romlist}
\end{letlist}

\medskip
\noindent
{\bf Assume (a) holds.}
Since $\Ga$ is finitely generated, so is $\sQ$. 

\medskip\noindent
\emph{Suppose $\sQ$ is infinite.}
We shall use the fact that $\sQ\in \EGF_{\alpha-1}$. Let $S$ be a finite set of generators of $\Ga$
with $S=S^{-1}$ and $\id\notin S$,
and let $G=G(\Ga,S)$ be the corresponding  Cayley graph of $\Ga$. 
A locally finite Cayley graph $G_{\sQ}$ of $\sQ$ may be constructed as follows. 
Let
\begin{equation*}
\ol{S}=\{\ol{s}=s\sN:  s\in S\},
\end{equation*}
be the (finite) generator set of $\sQ$ derived from $S$.
The vertex-set of $G_{\sQ}=G_\sQ(Q,\ol S)$ is the set of cosets $\{\ol v:= v\sN: v \in \Ga\}$,
and two such vertices $\ol{v}$, $\ol{w}$ are connected by an edge of $G_{\sQ}$ if and only if there exist  
$v\in\ol{v}$, $w\in\ol{w}$ such that $v$ and $w$ are connected by an edge in $G$. 

By H$_{\a-1}$, there exists $\ol\sH\normal\sQ$, not depending
on the choice of $S$, such that $G_\sQ$
admits a harmonic, \sghf\ $(h_\sQ,\ol\sH)$.
Let $h:\Ga\to\ZZ$ and $\sH\subseteq\Ga$ be given by
\begin{equation}\label{eq:4}
 h(v)=h_{\sQ}(\ol{v}),\qq
\sH=\bigcup_{\g\sN\in\ol\sH} \g\sN.
\end{equation}
The following lemma completes the proof of this case.

\begin{lemma}\label{lem:2}
We have that:
\begin{letlist}
\item $\sH\normal\Ga$, and $\sH$ acts quasi-transitively on $G$ by left-multiplication,
\item the pair $(h,\sH)$ is a harmonic, \sghf\ of $G$.
\end{letlist}
\end{lemma}

\begin{proof}
(a) Since $\ol\sH\normal\sQ$, we have that $(a\sN)\ol\sH(a\sN)^{-1}=\ol\sH$ for $a\in \Ga$,
whence 
\begin{equation}\label{eq:3}
(a\g a^{-1})\sN \in \ol\sH\q \text{whenever}\q a,\g\in\Ga,\ \g\sN\in\ol\sH.
\end{equation}
It is elementary that, for $a\in \Ga$ and $\g_1\sN,\g_2\sN\in\ol\sH$,
\begin{equation}\label{eq:5}
(a\g_1 a^{-1})\sN = (a\g_2a^{-1})\sN \q \text{if and only if}\q \g_1\sN=\g_2\sN. 
\end{equation}

Since $\ol\sH$ is a group, so is $\sH$. For $a\in \Ga$, by \eqref{eq:4}--\eqref{eq:5},
\begin{align*}
a\sH a^{-1}&=\bigcup_{\g \sN \in \ol\sH} a(\g\sN)a^{-1}
=\bigcup_{\g\sN\in\ol\sH} (a\g a^{-1})\sN\\
&=\bigcup_{\g\sN\in\ol\sH} \g \sN = \sH.
\end{align*}
Therefore, $\sH\normal \Ga$. We prove next that $[\Ga:\sH]<\oo$.

Since $(h_\sQ,\ol\sH)$ is a \ghf, we have that $[\sQ:\ol\sH]<\oo$.
Let $\ol{W}_1, \ol{W}_2,\dots,\ol{W}_k$ be the cosets of $\ol{\sH}$ in $\sQ$,
and let 
\begin{equation*}
W_i=\bigcup_{\g\sN \in \ol W_i} \g\sN.
\end{equation*}
We show next that each $W_i$ is contained in an orbit of $\sH$ acting
on $\Ga$.  (Actually the $W_i$ \emph{are} the orbits.) It follows that $\sH$ acts
quasi-transitively on $G$.

Without loss of generality, let $u,v \in W_1$. We shall show that there exists
$\nu\in\sH$ such that $v=\nu u$.
Suppose $u \in a\sN$, $v \in b\sN$ where $a\sN, b\sN  \in \ol W_1$. 
There exists $\g\sN\in\ol\sH$ such that $\g\sN a\sN=b \sN$, which is to
say that  $a\sN b^{-1}\in\ol\sH$.

There exist $n_i$ such that $u=an_1$, $v=b n_2$. Then, $u=(an_1n_2^{-1}b^{-1})v$,
and $\nu:=a(n_1n_2^{-1})b^{-1}\in \sH$ by \eqref{eq:4}.

(b) It is trivial that $h(\id)=h_{\sQ}(\ol\id)=0$.
For $\g\in \sH$ and $u,v\in \Ga$, we have
\begin{alignat*}{2}
h(\g u)-h(\g v)&=h_{\sQ}(\ol{\g}\ol{u})-h_{\sQ}(\ol{\g}\ol{v})\\
&=h_\sQ(\ol\g\, \ol u)-h_\sQ(\ol\g\, \ol v)\q&&\text{since $\sN$ is normal}\\
&=h_{\sQ}(\ol{u})-h_{\sQ}(\ol{v}) &&\text{since $h_\sQ$ is $\ol \sH$-difference invariant, 
$\ol\g\in\ol\sH$}\\
&=h(u)-h(v).
\end{alignat*}
Therefore, $h$ is \hdi.

For $v\in \Ga$, there exist $\ol{s}_1, \ol{s}_2\in \ol S$ such that
\begin{equation*}
h_{\sQ}(\ol{v}\,\ol{s}_1)<h_{\sQ}(\ol{v})<h_{\sQ}(\ol{v}\,\ol{s}_2),
\end{equation*}
whence, since $\sN$ is a normal subgroup of $\Ga$,
\begin{equation*}
h(vs_1)<h(v)<h(vs_2).
\end{equation*}
In conclusion, $(h,\sH)$ is a \sghf\ of $G$.

We show finally that $h$ is harmonic on the Cayley graph $G=(V,E)$.  The edges
incident to the vertex labelled $\g \in \Ga$ have the form $\langle \g, \g s\rangle$ for $s \in S$. 
Since $h_\sQ$ is harmonic on the quotient graph,
it suffices to show that the cardinality $N_s:=|\pd\g \cap (\g\sN s)|$ does not 
depend on the choice of $s\in 
S \setminus \sN$. For $s\in S\setminus \sN$ and  $n \in \sN$, $\g \sim \g ns$ if and only if
$ns\in S$, which is to say that $n \in S s^{-1}$, whence $N_s=|S|$.
\end{proof}

\medskip
\noindent
\emph{Suppose $\sQ$ is finite.}
Since $\sN\simeq\sN'\in\EGF_{\a-1}$, we have that $\sN\in \EGF_{\alpha-1}$ and $1<[\Ga:\sN]<\oo$. 
By H$_{\a-1}$, there exists $\sH'\normal\sN$ with $[\sN:\sH']<\oo$ such that any
locally finite Cayley graph $G_{\sN}$ of $\sN$
admits a \sghf\ of the form $(h_{\sN},\sH')$. 

Since $|\Ga/\sH'|=|\Ga/\sN|\cdot|\sN/\sH'|<\oo$, there exists (by Poincar\'e's Theorem for
subgroups) a subgroup $\sH\le \sH'$ that is normal in $\Ga$ with finite index, that is,
$\sH\normal \Ga$ and $1<[\Ga:\sH]<\oo$. Choose a locally finite Cayley graph $G_\sN$ of
$\sN$, and find a \sghf\ of the form $(h_\sN,\sH')$. Let $F:\sH\to\ZZ$ be the restriction
of $h_\sN$ to $\sH$. 

\begin{lemma}\label{lem:GHF}
The function $F$ is a \GHF\ on the group $\sH$.
\end{lemma}

\begin{proof}
As noted in \cite[Remark 4.2]{GL-Cayley}, a \GHF\ is 
a homomorphism from $\sH$ to $\ZZ$ that is not identically zero. For $\g_1,\g_2\in\sH$,
\begin{align*}
F(\g_1\g_2)-F(\g_1) &= h_\sN(\g_1\g_2)-h_\sN(\g_1)\\
&=h_\sN(\g_2)-h_\sN(\id)=F(\g_2),
\end{align*}
since $\g_1\in\sH'$ and $h_\sN$ is $\sH'$-difference invariant.
Therefore, $F$ is a homomorphism.  

It suffices now to show that $F\not\equiv 0$ on $\sH$.
Assume the converse, that $F \equiv 0$ on $\sH$.
For $\g\in\Ga$, there exists $a_\g\in \sN$ such that $\g\in a_\g \sH$,
so that $\g=a_\g\nu$ for $\nu\in\sH$. Since $h_\sN$ is
$\sH$-difference-invariant,
\begin{equation}\label{eq:6}
h_\sN(\g) = h_\sN(a_\g) + F(\nu) = h_\sN(a_\g).
\end{equation}
Now $|\sN/\sH|<\oo$, so we may restrict consideration to only finitely many $a_\g$.
Therefore, $h_\sN(\g)$ is bounded, which is impossible since $h_\sN$ is
a \ghf. We deduce that $F\not\equiv 0$ on $\sH$.
\end{proof}

Let $G=G(\Ga,S)$ be a locally finite Cayley graph of $\Ga$.
The triple $(\Ga,\sH,F)$ satisfies the conditions
of \cite[Thm 3.5]{GL-Cayley} with $\sH$ acting by left multiplication, and it follows that $G$ possesses
a harmonic \ghf\ of the form $(h,\sH)$.
 
\medskip\noindent
{\bf Assume (b) holds.}
Let $\Ga$ be finitely generated with finite generator set $S=\{s_1,s_2,\dots,s_k\}$. 
Since $\Ga=\bigcup_{\la\in \La}S_{\lambda}$, there exists $\la_i\in \La$ such that $s_i\in S_{\la_i}$. Let $L=\max\{\lambda_1,\la_2,\dots,\lambda_k\}$, so that $S_{L}=\Ga$. 
Then $\Ga\in \EGF_{\alpha-1}$, which contradicts the
minimality of $\a$.  
\end{proof}

\section{Criteria for the absence of height functions}\label{sec:noht}

This section contains some observations relevant to proofs 
in Sections \ref{sec:pfgrig}--\ref{sec:pfhig} of the non-existence
of \ghf s.

Let $\Ga=\langle S \mid R\rangle$ where $|S|<\oo$, and
let $G=(V,E)$ be the corresponding Cayley graph. Let
$\Pi$ be the set of permutations of $S$ that preserve $\Ga$
up to isomorphism, and write $e \in \Pi$ for the identity. 
Thus, $\pi\in\Pi$ acts on $\Ga$ by: for $w=s_1s_2\cdots s_m$
with $s_i\in S$, we have $\pi(w)=\pi(s_1)\pi(s_2)\cdots \pi(s_m)$.
It follows that  $\Pi\subseteq \Aut(G)$.
For $\g=g_1g_2\cdots g_n\in\Ga$ with $g_i \in S$, and $\pi\in\Pi$, 
we define $\g\pi\in\Aut(G)$ by $\g\pi(w)=g_1g_2\cdots g_n \pi(w)$,
$w\in V$. Write $\Ga\Pi\subseteq \Aut(G)$ for the subgroup containing all such $\g\pi$,
and note that $\g e$ operates on $G$ in the manner of
$\g$ with left-multiplication.

The \emph{stabilizer} $\Stab_v$ of  $v\in V$
is the set of automorphisms of $G$ that preserve $v$, that is,
$$
\Stab_v = \{\eta\in\Aut(G): \eta(v)=v\}.
$$

\begin{proposition}\label{thm:stab}
Suppose $\Stab_\id=\Pi$. 
\begin{letlist}
\item $\Aut(G) = \Ga\Pi$.
\item If $\sM \normal \Aut(G)$ has finite index,  the subgroup
 $\sN = \sM\cap \Ga$ satisfies $\sN \normal \Ga$ and $[\Ga:\sN]<\oo$.
\item If $G$ has a \ghf, then it has a \sghf.
\end{letlist}
\end{proposition}

\begin{proof}
Assume $\Stab_\id=\Pi$. 

(a)
Let $\eta\in\Aut(G)$, and write $\g=\eta(\id)$. Then $\g^{-1}\eta\in\Stab_\id$,
which is to say that $\g^{-1}\eta=\pi\in\Pi$, and thus
$\eta=\g\pi \in \Ga\Pi$ so that $\Aut(G)=\Ga\Pi$. Note for future use that
$$
[\Aut(G):\Ga] = |\Pi| < \oo.
$$

(b) 
Let $\sM\normal \Aut(G)$ be a finite-index normal subgroup,
and let $\sN=\{\g e: \g e \in \sM\}$. Viewed as automorphisms,
we have that $\g e= \g$, and hence $\sN \le \Ga \le \Aut(G)$.
For $\a\in\Ga$, $\nu\in \sN$, we have that $(\a^{-1}\nu\a) e =\a^{-1}(\nu e)\a \in \sM$, 
since $\sM \normal \Aut(G)$. Therefore, $\sN \normal \Ga$.

Since $\Ga,\sM\le \Aut(G)$ and $\sN =\Ga \cap \sM$, we have that
$$
[\Aut(G):\sN] \le [\Aut(G):\Ga]\cdot[\Aut(G): \sM]<\oo,
$$
which implies $[\Ga:\sN]<\oo$, as required.

(c) Let $(h,\sH)$ be a \ghf\ of $G$. Since $\sH$ is a finite-index normal subgroup of $\Aut(G)$,
by part (b), there exists $\sN \le \sH$ that is a finite-index normal subgroup of $\Ga$.
Since $\sN \le \sH$, $h$ acts on $\Ga$ and is $\sN$-difference invariant,
whence $(h,\sN)$ is a \sghf. 
\end{proof}

\begin{corollary}\label{cor}
Let $\Ga=\langle S\mid R\rangle$ have Cayley graph $G$ satisfying $\Stab_\id = \Pi$.
\begin{letlist}
\item  If $\Ga$ has no proper, normal subgroup with finite index, any \ghf\ of
$G$ is also a \GHF\ of $\Ga$.

\item  If every element in $\Ga$ has finite order, then $G$ has no \ghf.
\end{letlist}
\end{corollary}

\begin{proof}
(a) Let $(h,\sM)$ be a \ghf\ of $G$.
If $\Ga$ satisfies the given condition then, by Proposition \ref{thm:stab}(b), $\sM \supseteq \Ga$.
Therefore, $(h,\Ga)$ is a \ghf\ and hence a \GHF.

(b) If $G$ has a \ghf, by Proposition \ref{thm:stab}(c), $G$ has a \sghf\ $(h,\sN)$.
Assume every element of $\Ga$ has finite order.
For $\g\in \sN$ with $\g^n=\id$,
we have that $h(\g^n)=nh(\g)=0$, whence $h \equiv 0$ on $\sN$.

We now use the argument around \eqref{eq:6}. For $\g\in\Ga$, find $\a_\g$ such that $\g\in\a_\g\sN$.
Since $h$ is $\sN$-difference-invariant, there exists $\nu\in \sN$ such that
\begin{equation}\label{eq:6.2}
h(\g) =h(\a_\g) + h(\nu)=h(\a_\g).
\end{equation}
Now $[\Ga:\sN]<\oo$, so we may consider only finitely many choices for $\a_\g$. Therefore,
$h$ is bounded on $\Ga$, in contradiction of the assumption that it is a \ghf.
\end{proof}

\section{Proof of Theorem \ref{grig}} \label{sec:pfgrig}

The main step is to show that 
\begin{equation}\label{eq:triv}
\Stab_\id=\{e\},
\end{equation}
where $e$ is
the identity of $\Aut(G)$. Once this is shown, claim (a) follows from
Corollary \ref{cor}(b) and the fact that every element of the Grigorchuk
group has finite order, \cite{dlH}.  
It therefore suffices for (a) to show \eqref{eq:triv}, and
to this end we study the structure of the Cayley graph $G=(V,E)$.

It was shown in \cite{L85} (see also \cite[eqn (4.7)]{RG05}) that $\Ga = \langle S \mid R\rangle$
where $S=\{a,b,c,d\}$, $R$ is the following  set of relations
\begin{align}\label{eq:relations}
\id&=a^2=b^2=c^2=d^2=bcd\\
&=\si^k((ad)^4)=\si^k((adacac)^4),
\qq k = 0,1,2,\dots,
\nonumber
\end{align}
and $\si$ is the substitution
$$
\si: \begin{cases} a \mapsto aca,\\
b \mapsto d,\\
c \mapsto b,\\
d \mapsto c.
\end{cases}
$$
It follows that the following, written in terms of the reduced generator set
$\{a, b, c\}$ after elimination of $d$, are valid relations:
\begin{equation}\label{eq:relations2}
\id = a^2=b^2=c^2=(bc)^2=(abc)^4=(ac)^8=(abcacac)^4=(acab)^8=(ab)^{16},
\end{equation}
(see also \cite[Sect.\ 1]{RG05}).
Note the asymmetry between $b$ and $c$ in that $ab$ (\resp, $ac$) has order $16$ (\resp, $8$).

Let
\begin{equation*}
V_n=\{v\in\Ga: \dist(v,\id)=n\},
\end{equation*}
where $\dist$ denotes graph-distance on $G$.
Since $G$ is locally finite, $|V_n|<\infty$. 
For $\eta\in \Stab_\id$, $\eta$ restricted to $V_n$ 
is a permutation of $V_n$. As illustrated in Figure \ref{fig:G},
$$
V_0=\{\id\},\qq V_1=\{a,b,c\},\qq V_2=\{ab,ac,ba,bc=cb, ca\}.
$$

\begin{figure}[htbp]
\centerline{\includegraphics*[width=0.35\hsize]{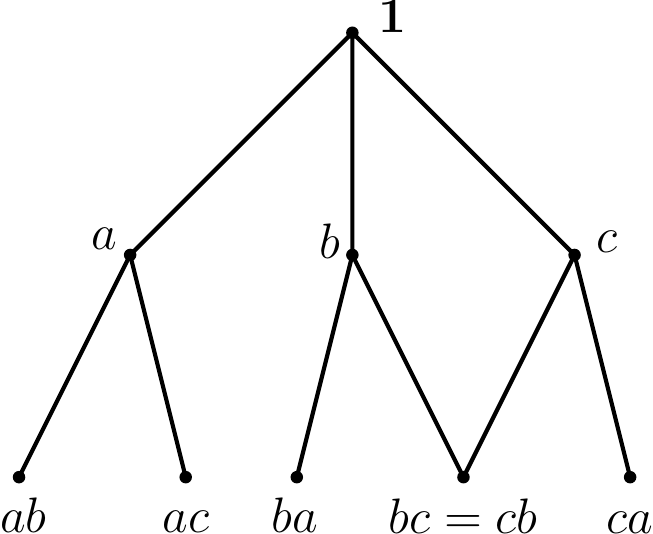}}
   \caption{The subgraph of $G$ on $V_0\cup V_1\cup V_2$.}
   \label{fig:G}
\end{figure}

Let $\eta\in\Stab_\id$, so that $\eta(a)\in V_1$. Since the shortest cycles using the edges 
$\langle\id,b\rangle$ and $\langle\id,c\rangle$ have length $4$, and using $\langle\id,a\rangle$
greater than $4$ (see Figure \ref{fig:G}), we have that
$\eta(a)=a$.
By a similar argument,  we obtain that, for $n \ge 1$,
\begin{equation}\label{eq:a}
\eta(va)=\eta(v)a, \qq v \in V_n,\ va\in V_{n+1},
\end{equation}
which we express by saying that $\eta$ maps $a$-type edges to $a$-type edges.

We show next that
\begin{equation}\label{eq:b}
\eta(vc)=\eta(v)c, \qq v \in V,\ \eta\in\Stab_\id,
\end{equation}
which is to say that $\eta$ maps $c$-type edges to 
$c$-type edges. By \eqref{eq:a}--\eqref{eq:b}, $\eta\in \Stab_\id$ maps
$b$-type edges to $b$-type edges also, whence $\eta=e$ as required. It remains to
prove \eqref{eq:b}.

Assume, in contradiction of \eqref{eq:b}, 
that there exists $v\in V$, $\eta\in \Stab_\id$ such that $\eta(vc)=\eta(v)b$. 
Since $ac$ has order $8$, we have that $(ca)^8=\id$.
Let $C$ be the directed cycle corresponding to the word $v (ca)^8$; thus, $C$ includes the edge 
$[ v,vc \rangle$. Then $\eta(C)$ is a cycle of length 
$16$ including the edge $[ \eta(v),\eta(v)b\rangle$. 
Since $C$ contains exactly $8$ $a$-type edges at alternating positions, by \eqref{eq:a},
so does $\eta(C)$.
Therefore, $\eta(C)$ has the form $\eta(v)ba\prod_{i=2}^{8}(x_i a)$, where $x_i\in\{b,c\}$ 
for $i=2,3,\dots,8$. In particular,
\begin{equation}\label{eq:ax}
ba\prod_{i=2}^{8}(x_ia)=\id, \qq x_i\in\{b,c\},\  i=2,3,\dots,8.
\end{equation}

The  word problem of the Grigorchuk group is solvable (see \cite{Grig84}
and \cite[Sect.\ 4]{RG05}),
in that there exists an algorithm to determine whether or not $w=\id$ for any 
given word $w\in \{a,b,c\}^*$ (where $S^*$ denotes the set of finite ordered sequences
of elements of $S$).
By applying this algorithm (see below), we deduce that \eqref{eq:ax}
has no solution. Equation \eqref{eq:b} follows, and the proof of part (a)
is complete.

Finally, here is a short amplification of the analysis of \eqref{eq:ax}.
The word in \eqref{eq:ax} has the form  $b(ay_1a)z_1(ay_2a)z_2(ay_3a)z_3(ay_4a)$, where
$y_i,z_j \in\{b,c\}$. By \eqref{eq:grigrels}, the effect of such a word on the right sub-tree $T_1$
is $\g_1:=ca(c/d)a(c/d)a(c/d)a$, where each term of the form $(y/z)$ is to be interpreted as 
\lq either $y$ or $z$'.
The effect of $\g_1$ on the left sub-tree $T_{10}$ of $T_1$ is $\g_{10}:=
a(d/b)(a/e)(d/b)$. If there is an odd number of appearances
of $a$ in $\g_{10}$, then $\g_{10}$ is not the identity, and thus we may assume  $\g_{10}:= a(d/b)a(d/b)$.
It is immediate that none of the four possibilities is the identity, and the claim follows.

Part (b) holds as follows. Suppose there exists $\sH\normal \Ga$, $\g\in\Ga$, and a non-constant
\hdi\  function $F:\g\sH\to\ZZ$.   It is elementary that $\sH$ is unimodular and
symmetric (see, for example, \cite[Sect.\ 4]{GrL3}). 
By \cite[Thm 3.5]{GL-Cayley} and the comment near the beginning
of \cite[Sect.\ 8]{GL-Cayley}, $G$ has a \ghf, in contradiction of part (a).

\section{Proof of Theorem \ref{thm:hg1}}\label{sec:pfhg}

We shall prove three statements:
\begin{romlist}
\item $\Ga$ has no \GHF,
\item $\Pi$ is
the cyclic group generated by the permutation $(abcd)$, with the convention that $\eta(x^{-1})
= \eta(x)^{-1}$, for $\eta\in\Pi$, $x\in\{a,b,c,d\}$,
\item $\Stab_\id=\Pi$.
\end{romlist}
It is proved in \cite{Hg} that the Higman group has
no proper, finite-index, normal subgroup, and
the result follows from the above statements by Corollary \ref{cor}(a).

\emph{Proof of (i).}
The absence of a \GHF\ is immediate by \cite[Example 6.3]{GL-Cayley}.

\emph{Proof of (ii).}
Evidently, $\Pi$ contains the given cyclic group, and we turn to the converse.
Since elements of $\Pi$ preserve $\Ga$ up to isomorphism,
\begin{equation}\label{eq:345}
\eta(x^{-1})=\eta(x)^{-1}, \qq x \in S.
\end{equation}
We next rule out the possibility that $\eta(x)=y^{-1}$ for some $x,y\in\{a,b,c,d\}$.
Suppose, for illustration, that $\eta(a)=b^{-1}$. By \eqref{eq:345}, the relation $a^{-1}ba=b^2$
becomes $b\be b^{-1}=\be^2$ where $\be=\eta(b)$. The Higman group has no such relation
with $\be \in S$. In summary,
\begin{equation}\label{eq:346}
\eta(x)\in\{a,b,c,d\},\q \eta(x^{-1})=\eta(x)^{-1}, \qq x \in \{a,b,c,d\}.
\end{equation}

The shortest cycles containing the edge $\langle \id,a\rangle$, 
modulo rotation and reversal, arise
from the relations
$ab^2 a^{-1}b^{-1}=\id$  and $ada^{-2}d^{-1}=\id$
(see Figure \ref{fig:BS-sheet}).
The first uses $a^{\pm1}$ twice and $b^{\pm1}$ thrice, and the second uses
$a^{\pm1}$ thrice and $d^{\pm1}$ twice. Let $\eta\in\Pi$, and suppose for illustration that 
$\eta(a)=b$ (the same argument is valid for any $\eta(x)$, 
$x\in\{a,b,c,d\}$). 
By considering the cycles starting $\langle\id,b\rangle$, 
$\langle\id,c\rangle$, $\langle\id, d\rangle$, and using \eqref{eq:346}, we deduce that
$$
\eta(b)=c, \q \eta(c)=d, \q \eta(d)=a,
$$
and the claim is proved.

\begin{figure}[htbp]
\centerline{\includegraphics*[width=0.6\hsize]{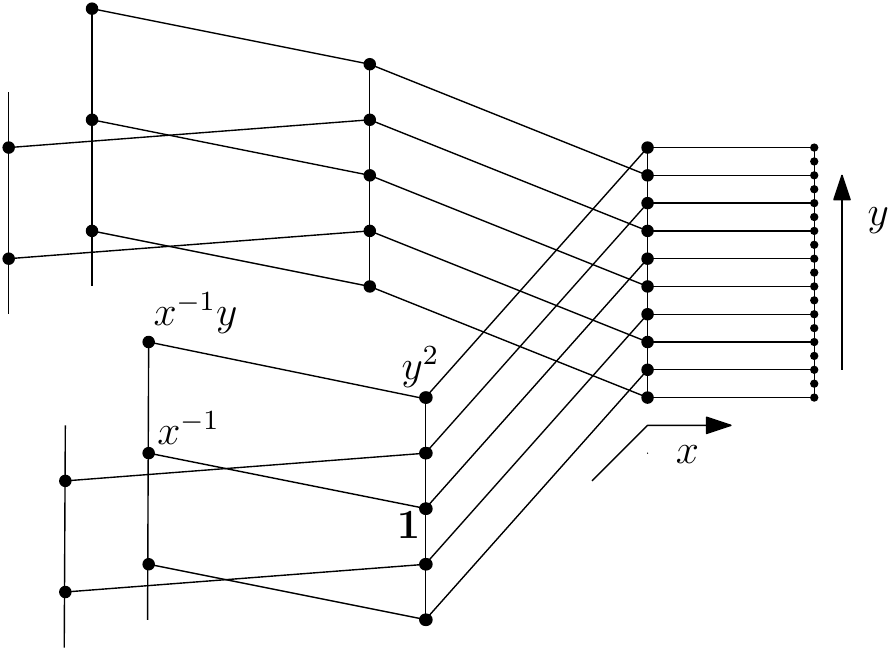}}
   \caption{Part of the Cayley graph 
   of the Baumslag--Solitar group $\BS(x,y)$.}
   \label{fig:BS-Cayley}
\end{figure}

\emph{Proof of (iii).}
We begin with some observations concerning the Baumslag--Solitar ($\BS$) group $\BS(x,y)$
with presentation $\langle x,y,x^{-1}, y^{-1} \mid x^{-1}yx=y^2\rangle$,
of which the Cayley graph is sketched in Figure \ref{fig:BS-Cayley}.
Edges of the form $\langle\g,\g x^{\pm 1}\rangle$ have \emph{type} $x$,
and of the form $\langle\g,\g y^{\pm 1}\rangle$  \emph{type} $y$. By inspection, 
the shortest cycles have length $5$ (see Figure \ref{fig:BS-sheet}), 
and, for $\g\in\BS(x,y)$,
\begin{align}\label{eq:22}
&\text{for $p,q=\pm 1$, the edges $\langle\g,\g x^p\rangle$ and $\langle\g,\g y^q\rangle$ 
lie in a common $5$-cycle,}\\
&\text{the third edge of any directed $5$-cycle beginning $[ \g, \g x\rangle$ has type $y$,}\label{eq:22.5}\\
&\text{the third edge of any directed $5$-cycle beginning $[ \g,\g x^{-1}\rangle$ has type $x$,}\label{eq:22.6}\\
&\text{every $5$-cycle contains two consecutive edges of type $y$, and not of type $x$,}\label{eq:22.7}\\
&\text{a type $x$ (\resp, type $y$) edge lies in $2$ (\resp, $3$) 5-cycles.}
\label{eq:23}\end{align}

\begin{figure}[htbp]
\centerline{\includegraphics*[width=0.44\hsize]{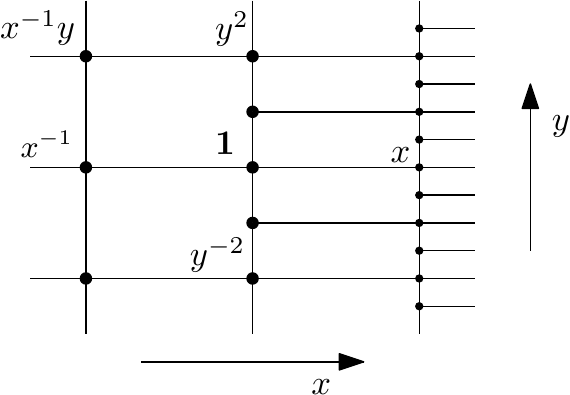}}
   \caption{Part of one `sheet' of the Cayley graph of $\BS(x,y)$.}
   \label{fig:BS-sheet}
\end{figure}

Returning to the Higman group, for convenience, we relabel the vector
$(a,b,c,d)$ as $(s_0,s_1,s_2,s_3)$,
with addition and subtraction of indices modulo $4$.
Let $G$ be the Cayley graph of the Higman group $\Ga=\langle S\mid R\rangle$, rooted at $\id$. An edge
of $G$ is said to be of \emph{type} $s_i$ if it has the
form $\langle \g,\g s_i^{\pm 1}\rangle$ with $\g\in\Ga$. We explain next how
to obtain  information about the types of the edges of $G$, by examination of $G$ only, 
and without further information about the vertex-labellings as elements of $\Ga$.

We consider first the set $\pde \id$ of edges of $G$ incident to $\id$.
Let $e_1=\langle \id,v\rangle $, $t\in\{0,1,2,3\}$, and $p\in\{-1,1\}$. 
Assume that 
\begin{equation}\label{eq:ass1}
v=s_t^ p,
\end{equation}
so that, in particular, $e_1$ has type $s_t$.
By \eqref{eq:22}, for $j=\pm 1$, $e_1$ lies in a $5$-cycle of
$\BS(s_{t-1},s_t)$ (\resp, $\BS(s_t,s_{t+1})$) containing $\langle\id, s_{t-1}^j\rangle$
(\resp, $\langle\id, s_{t+1}^j\rangle$).
On the other hand, by consideration of the relator set $R$, 
$e_1$ lies in no $5$-cycle including an edge of type $s_{t+2}$.
Therefore, the edges  of the form $\langle \id , s_{t+2}^{\pm 1}\rangle$
may be identified by examination of $G$, 
and we denote these as $g_1$, $g_2$. 
There is exactly one further edge of $\pde\id$ that lies in no $5$-cycle containing
either $g_1$ or $g_2$, and we denote this edge as $e_2$.
In summary, 
$$
\{e_1,e_2\}=\bigl\{\langle \id,s_t^{-1}\rangle,\langle \id,s_t\rangle\bigr\},\qq
\{g_1,g_2\} =\bigl\{\langle \id,s_{t+2}^{-1}\rangle, \langle \id,s_{t+2}\rangle\bigr\}.
$$ 

Having identified the edges of $\pde\id$ with types $s_t$ and $s_{t+2}$,
we move to the other endpoint $v=s_t^p$ of $e_1$, and apply the same argument.
Let $e_1$, $e_1'$ be the two type-$s_t$ edges incident to $v$.

We turn next to the remaining four edges of $\pde\id$. Let $k$ be such an edge, and
consider the property: \emph{$k$ lies in a $5$-cycle of $G$ containing both $e_1$
and $e_1'$}.
By \eqref{eq:22.7} and examination of
the Cayley graphs of the four groups $\BS(s_i,s_{i+1})$, $0\le i < 4$,
we see that $k$ has this property if it has type $t-1$, and not if it has type  $t+1$.
Thus we may identify the types of the four remaining edges of $\pde\id$, which
we write as 
$$
\{f_1,f_2\}=\bigl\{\langle \id,s_{t+1}^{-1}\rangle,\langle \id,s_{t+1}\rangle\bigr\},\qq
\{h_1,h_2\} =\bigl\{\langle \id,s_{t+3}^{-1}\rangle, \langle \id,s_{t+3}\rangle\bigr\}.
$$

Having determined the types of edges in $\pde\id$ 
(relative to the type $t$ of the initial edge $e_1$),
we move to an endpoint of  such an edge other than $\id$, and apply the same argument.  By 
iteration, we deduce the types of all edges of $G$. Let $T(k)$ denote the type of
edge $k$. It follows from the above that 
\begin{equation}\label{eq:rot}
\text{$T(k)-T(e_1)$ is independent of $t=T(e_1)$},
\end{equation}
with arithmetic on indices, modulo $4$.

We explain next how to identify the value of $p=p(v)$ in \eqref{eq:ass1} from the
graphical structure of $G$.
Let $S_i$ be the subgraph of $G$
containing all edges with type either $s_i$ or $s_{i+1}$, so that 
each component of $S_i$ is isomorphic to the Cayley
graph of $\BS(s_i,s_{i+1})$.  By \eqref{eq:22.5}--\eqref{eq:22.6},
every directed $5$-cycle of $\BS(s_t,s_{t+1})$  starting with the edge 
$[\id, s_t\rangle$  has third edge with type $s_{t+1}$, 
whereas every directed $5$-cycle starting with  
$[\id, s_t^{-1}\rangle$  has third edge with type $s_t$.
We examine $S_t$ to determine which of these two cases holds,
and the outcome determines the value of $p=p(v)$.

The above argument is applied to each directed edge $[\g,\g s_i^{\pm 1}\rangle$
of $G$, and the power of $s_i$ is thus determined from the graphical structure of $G$.

Let $\eta\in\Stab_\id$. By \eqref{eq:rot}, the effect of $\eta$ is to
change the edge-types by 
$$
T(k) \mapsto T(k) + T(\eta(e_1))-t.
$$ 
Now, $\eta(v)$ is adjacent to $\id$ and, by the above, once $\eta(v)$ is known,
the action of $\eta$ on the rest of $G$ is determined.  Since $\eta\in\Aut(G)$,
$\eta(v)$ may be any neighbour $w$ of $\id$ with the property that $p(w)=p(v)$.
There are exactly four such neighbours (including $v$)
and we deduce from \eqref{eq:rot} that $\eta$ lies in the cyclic group
generated by the permutation $(s_0s_1s_2s_3)$.

\section{Proof of Theorem \ref{thm:hg2}}\label{sec:pfhig}

We shall prove three statements:
\begin{romlist}
\item $\Ga$ has no \GHF,
\item $\Stab_\id=\Pi$ where $\Pi=\{e\}$,
\item $\Ga$ has no proper normal subgroup with finite index.
\end{romlist}
The result follows from these statements by Corollary \ref{cor}(a), and we turn to their proofs.

\emph{Proof of (i).}
The absence of a \GHF\ is immediate by \cite[Thm 4.1(b)]{GL-Cayley}.

\emph{Proof of (ii).}
Let $\eta\in\Stab_\id$ and $\g \in \Ga$. We consider the action of $\eta$ on directed edges of $G$.
By inspection of the set $R'$ of relations, an edge of the type $\langle \g, \g x\rangle$
lies in shortest cycles of length
\begin{equation*}
\begin{cases} 5,\ 8 &\text{if } x=a^{\pm 1},\\
5,\ 7 &\text{if } x=b^{\pm 1},\\
7,\ 8 &\text{if } x=c^{\pm 1},\\
9,\ 11 &\text{if } x=d^{\pm 1}.
\end{cases}
\end{equation*}
Since the four combinations are distinct, it must be that
\begin{equation}\label{eq:230}
\eta([\g,\g x\rangle) = [\g',\g'x^{\pm 1}\rangle, \qq \g\in\Ga, \ x\in S,
\end{equation}
where $\g'=\eta(\g)$.
We  show next that
\begin{equation}\label{eq:213}
\eta([\g,\g x\rangle) \ne [\g',\g' x^{-1}\rangle, \qq \g\in\Ga,\ x\in S,
\end{equation} 
which combines with \eqref{eq:230} to imply $\eta=e$ as required.

It suffices to consider the case $x=a$ in \eqref{eq:213},
since a similar proof holds in the other cases.
Suppose $\eta([\g,\g a\rangle)=[\g',\g' a^{-1}\rangle$, and consider the 
cycle corresponding to $\g ab^{-2}a^{-1}b^{-1}$, that is $(\g,\g a,\g a b^{-1},\g ab^{-2}, \g ab^{-2}a^{-1},
\g ab^{-2}a^{-1}b^{-1}=\g)$. 
By \eqref{eq:230}, this is mapped under $\eta$ to the cycle 
corresponding to  $\g' a^{-1}b^{\pm 2}a^{\pm 1}b^{\pm 1}$.
By examining the relation set $R'$, the only cycles beginning $\g' a^{-1}b^{\pm 1}$ with length 
not exceeding $5$ are $\g'a^{-1}bab^{-2}$ and $\g'a^{-1}b^{-1}ab^2$, in contradiction
of the above (since the third step of these two cycles is $a$ rather than the required $b^{\pm 1}$).

\emph{Proof of (iii).}
Suppose $\sN$ is a proper normal subgroup of $\Ga$ with finite index. 
The quotient group $\Ga/\sN$ is non-trivial and finite  with generators $\ol s=s\sN$,
$s\in S$, satisfying
\begin{equation}\label{eq:ord3}
\begin{alignedat}{3}
\ol{a}^{-1}\ol{b}\ol{a}&=\ol{b}^2,&\qq
\ol{b}^{-2}\ol{c}\ol{b}^2&&=\ol{c}^2,\\
\ol{c}^{-3}\ol{d}\ol{c}^3&=\ol{d}^2,&\qq
\ol{d}^{-4}\ol{a}\ol{d}^4&&=\ol{a}^{2}.
\end{alignedat}
\end{equation}
Since $\Ga/\sN$ is finite, each $\ol s$ has finite order, denoted $\Od(\ol{s})$.
It follows from \eqref{eq:ord3} that 
\begin{equation}\label{eq:ord4}
\Od(\ol s)>1, \qq s=a,b,c,d.
\end{equation}
To see this, suppose for illustration that $\Od(\ol c)=1$, so that $\ol c=\ol\id$. 
By the third equation of \eqref{eq:ord3},
$\Od(\ol d)=1$, so that $\ol d=\ol\id$, and similarly for $\ol a$ and $\ol b$, implying
that $\Ga/\sN$ is trivial, a contradiction.

By induction, for $n \ge 1$,
\begin{alignat*}{3}
\ol{a}^{-n}\ol{b}\ol{a}^n&=\ol{b}^{2^n},&\qq
\ol{b}^{-2n}\ol{c}\ol{b}^{2n}&&=\ol{c}^{2^n},\\
\ol{c}^{-3n}\ol{d}\ol{c}^{3n}&=\ol{d}^{2^n},&\qq
\ol{d}^{-4n}\ol{a}\ol{d}^{4n}&&=\ol{a}^{2^n},
\end{alignat*}
whence, by setting $n=\Od(\ol a)$, etc,
\begin{equation}\label{eq:ord2}
\begin{alignedat}{2}
\Od(\ol{b})&\bigmid (2^{\Od(\ol{a})}-1),\qq\Od(\ol{c})&&\bigmid (2^{\Od(\ol{b})}-1),\\
\qq\Od(\ol{d})&\bigmid (2^{\Od(\ol{c})}-1),\qq\Od(\ol{a})&&\bigmid (2^{\Od(\ol{d})}-1),
\end{alignedat}
\end{equation}
where $u \mid v$ means that $v$ is a multiple of $u$.
We shall deduce a contradiction from \eqref{eq:ord4} and \eqref{eq:ord2}.
This is done as in \cite{Hg}, of which we reproduce the proof for completeness.

Let $p$ be the least prime factor of the four integers $\Od(\ol s)$, $s\in\{a,b,c,d\}$. By \eqref{eq:ord4},
$p>1$. Suppose that $p \mid \Od(\ol a)$ (with a similar argument
if $p\mid \Od(\ol s)$ for some other parameter $s$).
Then $p \mid 2^{\Od(\ol d)}-1$ by \eqref{eq:ord2}, and in particular $p$ is odd and 
therefore coprime with $2$.
Let $r$ be the multiplicative order of $2$ mod $p$, that is,
the least positive integer $r$ such that $p \mid 2^r-1$.
In particular, $r > 1$, so that $r$ has a prime factor $q$.
By Fermat's little theorem, $r \mid p-1$ so that $q<p$.
Furthermore, $r \mid \Od(\ol d)$ so that $q \mid \Od(\ol d)$, in contradiction of the
minimality of $p$.
We deduce that $\Ga$ has no proper, normal subgroup with finite index.

\section*{Acknowledgements} 
GRG was supported in part
by the Engineering and Physical Sciences Research Council under grant EP/103372X/1. 
ZL acknowledges support from the Simons Foundation under  grant $\#$351813.

\bibliography{new1}
\bibliographystyle{amsplain}

\end{document}